\documentclass[11pt]{amsart} 

\usepackage[utf8]{inputenc} 
\newtheorem{theorem}{Theorem}[section]
\newtheorem{corollary}[theorem]{Corollary}
\newtheorem{main}{Main Theorem}
\newtheorem{lemma}[theorem]{Lemma}

\theoremstyle{definition}
\newtheorem{example}[theorem]{Example}
\newtheorem{definition}[theorem]{Definition}
\newtheorem{remark}[theorem]{Remark}

\usepackage{graphicx} 
\usepackage{physics}

\title {On partial maps derived from flows}
\author{Tomoharu Suda}
\address{RIKEN Center for Sustainable Resource Science, Japan}
\email{tomoharu.suda@riken.jp}
\begin{document}
\begin{abstract}
The first-return map, or the Poincar\'e map, is a fundamental concept in the theory of flows. However, it can generally be defined only partially, and additional conditions are required to define it globally. Since this partiality reflects the dynamics, the flow can be described by considering the domain and behavior of such maps. In this study, we define the concept of first-out maps and first-in maps, which are partial maps derived from flows, to enable such analysis. Moreover, we generalize some notions related to the first-return maps. It is shown that the boundary points of an open set can be classified based on the behavior of these maps, and that this classification is invariant under topological equivalence. Further, we show that some dynamical properties of a flow can be described in terms of the types of boundary points. In particular, if the flow is planar and the open set has a Jordan curve as its boundary, a more detailed analysis is possible. We present results on the conditions that restrict possible forms of the first-out maps. Finally, as an application of the results obtained, we consider the relationship between flows and a class of hybrid systems.
\end{abstract}
\maketitle 
\section{Introduction}

The first-return map, or the Poincar\'e map, is a fundamental tool in the study of continuous-time dynamical systems. It has a wide range of applications, from the classical proof of the Poincar\'e-Bendixon theorem to the analysis of chaotic dynamical systems \cite{robinson1998dynamical, brin2002introduction}, and is arguably one of the most valuable concepts in the theory of flows.  

In general, the first-return map is not necessarily defined globally, and we need additional assumptions to ensure its global existence. For example, we may construct a global Poincar\'e section to obtain a globally defined first-return map for a broad class of flows \cite{basener2004every}. Even then, there is still no guarantee that the domain of resulted maps is in good shape and suitable for analyzing dynamics. Since the suspension of the first-return map is topologically equivalent to the original flow \cite{irwin2001smooth, yang2000remark, suda2022}, the existence of a domain with good topological properties might restrict the possible motion of the flow. 

While this lack of global existence may be a hindrance to the applications, there is the information we can extract from it. Indeed, if such a map is not defined at a point, this fact itself conveys information on the asymptotic behavior of the orbit starting from there. Therefore it appears worth asking how much we can infer from this partiality of the first-return map.

Another question regarding the first-return map is the relationship with the hybrid systems, which are defined by a combination of discrete and continuous dynamical systems. For example, if we consider the bouncing motion of a ball, we obtain a hybrid system, which combines the continuous motion under gravity and the discrete change of velocity due to collision with the floor. 
It is known that some hybrid systems can be obtained from flows by ``squeezing" the phase space \cite{bernardo2008piecewise}. This construction is at least qualitatively similar to the Poincar\'e half map and the half-return map, which can be regarded as a partial version of the first-return map \cite{uehleke1983complicated, carmona2021integral, novaes2021lyapunov}. This motivates us to consider the problem of what kind of partial maps can be expressed in terms of flows via constructions similar to the first-return maps. 

%

The purpose of the present article is to consider partial maps derived from a flow in a manner similar to the first-return map and apply them to the analysis of flows. Further, we aim to apply the results obtained here to a converse problem of representing a partial map by a flow so that we reduce a class of hybrid systems to flows.

The main results of this study are as follows. First, we introduce the concept of first-out and first-in maps, which are generalizations of the first-return map in the usual sense (Definition \ref{def_exit}). For a flow $\Phi$ on X and an open set $A\subset X$, the first-out map is defined to be the map assigning each point $x \in \partial A$ to the point $\Phi(\tau, x)$, where $\tau = \inf\{t >0 \mid \Phi(t,x) \not\in A\}$ if $\tau$ is finite. Thus, the first-out map moves a point on $\partial A$ to the first intersection point of the boundary and its forward orbit other than itself. The first-out map is defined to be a partial map and we do not require it to be defined globally on $\partial A$. This is a direct generalization of the Poincar\'e half map. Similarly, the first-in map is defined for a closed set in terms of the first-return point. 

Here we note that the concept of first-out and first-in maps can be defined for general continuous flows on a topological space as it does not depend on the differentiable structure. This is an advantage over the classical notion of a first-return map based on the cross-section. For example, we may consider flows with non-differentiable points. Further, by explicitly allowing maps to be partial, it is not necessary to verify the existence of a returning orbit to apply the results.

We can classify boundary points of a regular open set, which has the property that its boundary coincides with that of the closure, into five types according to whether these maps are defined and whether they are fixed under these maps. For a regular open set $A$, let $E_A$ be the first-out map and $R_A$ be the first-in map. Then, each $x \in \partial A$ falls into one of the following types:
\begin{enumerate}
	\item (Type A-1, launching points) $E_A(x) = x$ and $R_{\bar A} (x) \neq x.$
	\item (Type A-2, diving points) $E_A(x) \neq x$ and $R_{\bar A} (x) = x.$
	\item (Type A-3, tangency points) $E_A(x) = x$ and $R_{\bar A} (x) = x.$
	\item (Type B, never-to-return points) $R_A(x)$ is undefined.
	\item (Type C, never-to-leave points) $E_A(x)$ is undefined.
\end{enumerate}
 Figure \ref{fig_type} shows a rough sketch of orbits for each type. 
 
 These types are invariant under topological equivalence and, therefore, can be used to describe the dynamics. More concretely, we show the following result.
\begin{main}
Let $(X, \Phi)$ and $(Y, \Psi)$ be flows topologically equivalent via a homeomorphism $h: X \to Y$, and $A \subset X$ and $B \subset Y$ be regular open sets with $B = h(A)$. Then, $x \in \partial A$ and $h(x) \in \partial B$ have the same type.
\end{main}
Further, we show that various dynamical properties of a flow can be described in terms of these types of boundary points. For example, the invariance of an open set can be expressed by specifying the types of boundary points, which may be regarded as a generalization of the result that a set is invariant if a vector field points inward on the boundary.

For planar flows and open sets with Jordan curves as boundaries, we can obtain a more convenient representation of the first-out maps in terms of parametrization. If $c:[0,1) \to \mathbb{R}^2$ is a parametrization of the boundary of an open set $A \subset \mathbb{R}^2$, then it induces a partial map $F_E: [0,1) \to [0,1)$ by $c\qty(F_E(s)) = E_A(c(s)),$ where $E_A$ is the first-out map of $A$. The map $F_E$ encodes information on $E_A$ in a form more suitable to analyze.

Since the first-out map is defined by a flow, its possible forms are restricted. In particular, we have the following monotonicity result for planar flows, which asserts that a parametric representation of an exit map should decrease locally around a point where it is not identity.

\begin{main}
If a parametric representation of a first-out map $F_E$ is continuous at $s$ and $F_E(s)\neq s$, there exists $\delta>0$ such that $F_E(t) > F_E(s)$ whenever $s-\delta < t <s$, and $F_E(t) < F_E(s)$ whenever $s<t<s+\delta.$
\end{main}

Thus, the first-out map induces a well-behaved map on $[0,1)$. A kind of converse to this result holds, and we can represent a map on $\mathbb{R}$ by a planar flow if it is sufficiently well-behaved. 

\begin{main}
Let $P: \mathbb{R} \to \mathbb{R}$ be a continuous map such that $P(-\infty, 0] = [0, \infty)$, $P(0) =0$, and $P$ is two-to-one except at $0$ and identity on $[0, \infty)$. Then, $\qty(P(x), 0) = E_{H^{-}}^\Phi(x,0)$ for some flow $\Phi$, where $E_{H^{-}}^\Phi$ is the first-out map for $H^{-} :=\{(x,y) \mid y <0\}.$
\end{main}
Using this result, we can represent a class of hybrid systems in terms of flows. Here we consider impacting systems, which are simple hybrid systems that consist of a flow and a resetting map. While an exact definition of an impact system is given in Definition \ref{def_imp}, let us introduce an example of such a system to illustrate it.
\begin{example}
The motion of a bouncing ball is formulated as an impacting system. Namely, its state is described as a point in the closure of the upper half plane $\overline{H^+} = \mathbb{R} \times [0, \infty),$ and its dynamics is given by
\[
	\dv{t} \mqty(x \\ y) = \mqty(y \\ -g)
\]
for $x >0$, where $g >0$ is the acceleration of gravity, and
\[
	y(t+0) = -r y(t-0)
\]
when $x(t) = 0$, where $r>0$ is the coefficient of restitution. Thus the dynamics are described by a flow on the upper half plane and a map on the y-axis.
Its trajectories can be defined as a kind of curve on $\overline{H^+} $ with discontinuities.
\end{example}
While an impacting system is determined by a flow and a map, we can simplify the map part of the system into that of the impact oscillator, namely, a $-1$ times map, if it is sufficiently well-behaved. Thus, there is a normal form for the map part.

\begin{main}
Let $(P, \Phi, \Phi_s)$ be an impacting system induced by local flows. If $P$ is defined on the whole $\mathbb{R}$, continuous, and not identity, then $(P, \Phi, \Phi_s)$ is topologically conjugate with another impacting system $(Q, \Psi, \Psi_s)$, where $Q(x) = -x$ if $x \leq 0.$
\end{main}  

This article is organized as follows. In Section 2, we introduce basic terms and some preliminary results. In Section 3, we define the notion of first-out maps and first-in maps, study their basic properties, introduce the concept of types of boundary points, and apply them to the description of dynamical properties. In Section 4, we consider the first-out maps of planar flows and their parametric representation. In Section 5, we apply the results obtained to the study of a class of hybrid systems.
\section{Preliminaries}
In this section, we describe some basic definitions and results used throughout this article.

First, we introduce the main objects of our consideration here, i.e., flows and partial maps.
\begin{definition}[Flow]
Let $X$ be a topological space. A continuous map $\Phi:\mathbb{R}\times X \to X$ is a \emph{flow} if
\begin{enumerate}
	\item For each $x\in X,$ $\Phi(0,x) = x.$
	\item For each $s,t\in \mathbb{R}$ and $x \in X,$ $\Phi\qty(s+t,x) = \Phi\qty(s, \Phi\qty(t, x)).$
\end{enumerate}
A flow $\Phi$ on $X$ is denoted by $(X, \Phi).$

For each $x \in X$, the \emph{forward orbit} of $x$ is the set $\mathcal{O}^{+}(x):=\{\Phi\qty(t,x) \mid t \geq 0\}.$ The \emph{backward orbit} of $x$ is defined similarly by
$\mathcal{O}^{-}(x):=\{\Phi\qty(t,x) \mid t \leq 0\}.$ The \emph{orbit} of $x$ is the set $\mathcal{O}(x):=\mathcal{O}^{+}(x) \cup \mathcal{O}^{-}(x).$
\end{definition}
Details on partial maps can be found in \cite{abd1980compact}.
\begin{definition}[Partial map]
Let $X$ and $Y$ be topological spaces.
A \emph{partial map} is a pair $(D, f)$ of subset $D \subset X$ and a map $f: D \to Y.$ For a partial map $(D,f)$, the set $D$ is called the \emph{domain}, and is denoted by $\mathrm{dom}\, f$. As a convention, we denote a partial map $(D, f)$ by $f: X \to Y.$ The \emph{image} of a partial map is the set $\mathrm{im}\, f := f \left(\mathrm{dom}\,f \right).$

A partial map $f:X \to Y$ is a \emph{partial map with open domain} if $\mathrm{dom}\,f$ is open and $f: \mathrm{dom}\,f \to Y$ is continuous.
\end{definition}

We now introduce some preliminary results.
The following lemma is a generalization of the intermediate value theorem.
\begin{lemma}\label{lem_IVT}
Let X be a topological space, $A \subset X$ an open subset, and $c: [\alpha, \beta] \to X$ be continuous. If $c(\alpha) \in A$ and $c(\beta) \not \in A$, then there exists $\gamma \in [\alpha, \beta]$ such that $c(\gamma) \in \partial A.$
\end{lemma}
\begin{proof}
Since $X = \bar A \cup (X\backslash A)$, we have $[\alpha, \beta] = c^{-1}\left( \bar A\right) \cup c^{-1}\left(X\backslash A\right).$ By the connectedness of $[\alpha, \beta]$, we have 
\[
	 c^{-1}\left( \bar A\right) \cap c^{-1}\left(X\backslash A\right)=  c^{-1}\left( \bar A \cap (X\backslash A) \right)\neq \emptyset.
\]
Since we have $\bar A \cap (X\backslash A) = \partial A$, there exists $\gamma \in [\alpha, \beta]$ such that $c(\gamma) \in \partial A.$
\end{proof}

The next lemma is trivial but useful in constructing a flow with prescribed properties.
\begin{lemma}\label{lem_mkflow}

Let $X$ and $Y$ be topological spaces and $\Phi: \mathbb{R}\times X \to X$ a continuous flow. If $h:X \to Y$ is a homeomorphism, there exists a unique flow $\Psi:  \mathbb{R}\times Y \to Y$ such that
\[
	\Psi(t, h(x)) = h\left(\Phi(t,x)\right)
\] 
for all $x \in X$ and $t \in \mathbb{R}.$
\end{lemma}

\begin{example}\label{ex_moe}
Lemma \ref{lem_mkflow} enables us to apply the idea of conformal transformation, which is often used in hydrodynamics, to general flows.
Let $\Phi: \mathbb{R}\times \mathbb{R}^2 \to \mathbb{R}^2$ be a continuous flow. 
Then, it can be shown that $\Phi$ can be extended to a continuous flow $\hat{ \Phi}$ on the Riemann sphere $\hat{\mathbb{C}}$ by setting $\hat{ \Phi}(t,\infty) = \infty$ for all $t \in \mathbb{R}.$ If $M_A: \hat{\mathbb{C}} \to \hat{\mathbb{C}}$ is the M\" obius transformation defined by the matrix $A$, there exists another flow $\hat{ \Psi}$ on the Riemann sphere such that $\hat{ \Psi}(t, M_A(z)) = M_A(\hat{ \Phi}(t,z))$ by Lemma \ref{lem_mkflow}. If $\infty$ is an equilibrium point of $\hat{\Psi}$, $\hat{\Psi}$ is an extension of a continuous flow on $\mathbb{R}^2.$ By finding a suitable M\"obius transformation, we may map the interior of the unit disc to the lower half-plane, for example.
\end{example}

\section{First-out maps and first-in maps}
In this section, we first introduce the notion of first-out maps and first-in maps and study their basic properties. By considering the domain and behavior of these maps, we define the types of boundary points of an open set that can be used to describe the dynamics of the flow from which the map was derived.
\subsection{Definition and basic properties}
First, we define first-out maps and first-in maps as follows.
\begin{definition}[First-out maps and first-in maps]\label{def_exit}
Let $(X, \Phi)$ be a flow.
\begin{enumerate}
	\item For an open set $A \subset X$, the \emph{first-out map} $E^\Phi_A: \partial A \to \partial A$ is a partial map defined by
	\[
		E^\Phi_A(x) := \Phi\qty(T^e_A(x), x),
	\]
	where
	\[
		T^e_A(x) := \inf\{t >0 \mid \Phi(t,x) \not\in A\}.
 	\]
	\item For a closed set $B \subset X$, the \emph{first-in map} $R^\Phi_B: \partial B \to \partial B$ is a partial map defined by
	\[
		R^\Phi_B(x) := \Phi\qty(T^r_B(x), x),
	\]
	where
	\[
		T^r_B(x):= \inf\{t >0 \mid \Phi(t,x) \in B\}.
 	\]
\end{enumerate}
For notational convenience, we drop the index for the flow and denote it as $E_A$ if there is no confusion.
\end{definition}
We need to check that the first-out map is well-defined.
\begin{lemma}
First-out maps are well-defined. That is,  $\mathrm{im}\,E_A \subset \partial A.$ 
\end{lemma}
\begin{proof}
Let $x \in \partial A,$ where $A$ is an open subset, and $t_0 := T^e_A(x)$. If $t_0= 0,$ then $E_A(x) = x \in \partial A$. 

If $t_0 > 0$, we have $\Phi(t,x) \in A$ for all $t \in (0, t_0)$, and there exists a sequence $t_n >0$ with $t_n \to t_0$ as $n \to \infty$ and $\Phi(t_n, x) \not \in A$. Since $A$ is open, we observe that $\Phi(t_0, x) \not \in A.$ By Lemma \ref{lem_IVT}, for all $\alpha \in (0, t_0),$ there exists $\gamma \in [\alpha, t_0]$, such that $\Phi(\gamma, x) \in \partial A$. By the definition of $t_0$, we have $\gamma \leq t_0$. Therefore we conclude that $\gamma = t_0$, and consequently, $E_A(x) \in \partial A$.
 \end{proof}
 
 \begin{remark}\label{rem_map}
Consequently, if $A$ is open, we have
\begin{enumerate}
	\item $E_A(x) \not \in A$ if $x \in \mathrm{dom}\, E_A,$
	\item $R_{\bar A}(x)  \in \bar A$ if $x \in \mathrm{dom}\, R_{\bar A}.$
\end{enumerate}
 \end{remark}
 \begin{remark}
 Note that $x \in \mathrm{dom}\, E_A$ if and only if $ \Phi(t,x) \not\in A$ for some $t>0.$ Equivalently, $x \not \in  \mathrm{dom}\, E_A$ if and only if $\mathcal{O}^+(x)\backslash \{x\} \subset A.$
 \end{remark}
 \begin{remark}
  In the definition of first-out or first-in maps, we do not require open sets or closed sets to be connected because this property is not necessary for defining them. 
 \end{remark}
The first-out map is a dual concept to the first-in map. This is observed by the next lemma, which follows immediately from the definition. Therefore, we will mainly consider the first-out map in what follows.
\begin{lemma}\label{lem_sym}
Let $(X, \Phi)$ be a flow, $A \subset X$ an open set, and $B \subset X$ a closed set. Then, the following hold identically.
\[
	\begin{aligned}
		E_A & = R_{X \backslash A}, \\
		R_B & = E_{X \backslash B}.
	\end{aligned}
\]
\end{lemma}

 Now we present an example of the first-out maps and first-in maps.
 
\begin{example}\label{ex_exmap}
Let us consider the flow on $\mathbb{R}^2$ generated by the vector field
\begin{equation}\label{ex_exit}
	v(x,y) := (x,-y).
\end{equation}
For the unit disc $D = \{(x,y)\mid x^2+y^2 <1\}$, the first-out map and the first-in map are given by
\[
	E_D(\theta) = \begin{cases}
					\theta & (0 \leq \theta < \frac{\pi}{4})\\
					\frac{\pi}{2}-\theta & ( \frac{\pi}{4} \leq \theta < \frac{\pi}{2})\\
					\text{undefined } & (  \theta = \frac{\pi}{2})\\
					\frac{3}{2} \pi -\theta & ( \frac{\pi}{2} < \theta \leq \frac{3}{4} \pi)\\
					\theta & (  \frac{3}{4} \pi < \theta \leq  \frac{5}{4} \pi)\\
					\frac{5}{2}\pi-\theta & ( \frac{5}{4} \pi <  \theta < \frac{3}{2}\pi)\\
					\text{undefined }& ( \theta = \frac{3}{2}\pi)\\
					\frac{7}{2}\pi-\theta & ( \frac{3}{2}\pi <  \theta \leq  \frac{7}{4} \pi)\\
					\theta & ( \frac{7}{4} \pi <  \theta < 2\pi)
					\end{cases}
\]

and 

\[
	R_{\bar D}(\theta) = \begin{cases}
					\text{undefined } & (0 \leq \theta \leq \frac{\pi}{4})\\
					\theta & ( \frac{\pi}{4} < \theta < \frac{3}{4} \pi)\\
					\text{undefined } & (   \frac{3}{4} \pi < \theta < \frac{5}{4} \pi)\\
					\theta & (\frac{5}{4} \pi < \theta < \frac{7}{4} \pi )\\
					\text{undefined } & (  \frac{7}{4} \pi \leq \theta <2 \pi)
					\end{cases}
\]
respectively, where $S^1 = \partial D$ is parametrized by the angle $\theta$. The plots of these partial maps are shown in Figure \ref{fig_ex}.

\begin{figure}[htbp]
	  \begin{minipage}[b]{0.45\linewidth}
    \centering
    \includegraphics[keepaspectratio, scale=0.25]{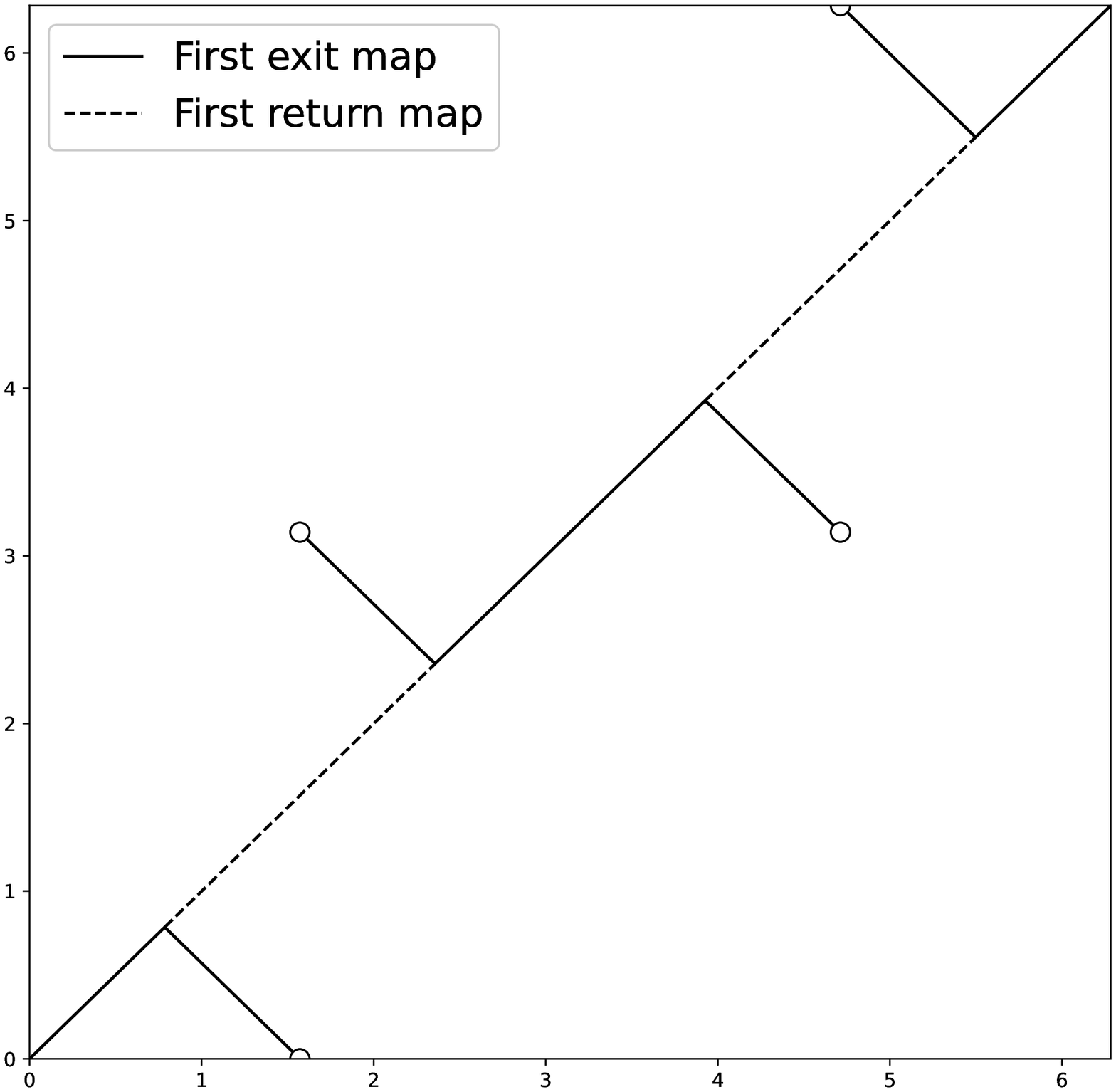}
  \end{minipage}
  \begin{minipage}[b]{0.45\linewidth}
    \centering
    \includegraphics[keepaspectratio, scale=0.25]{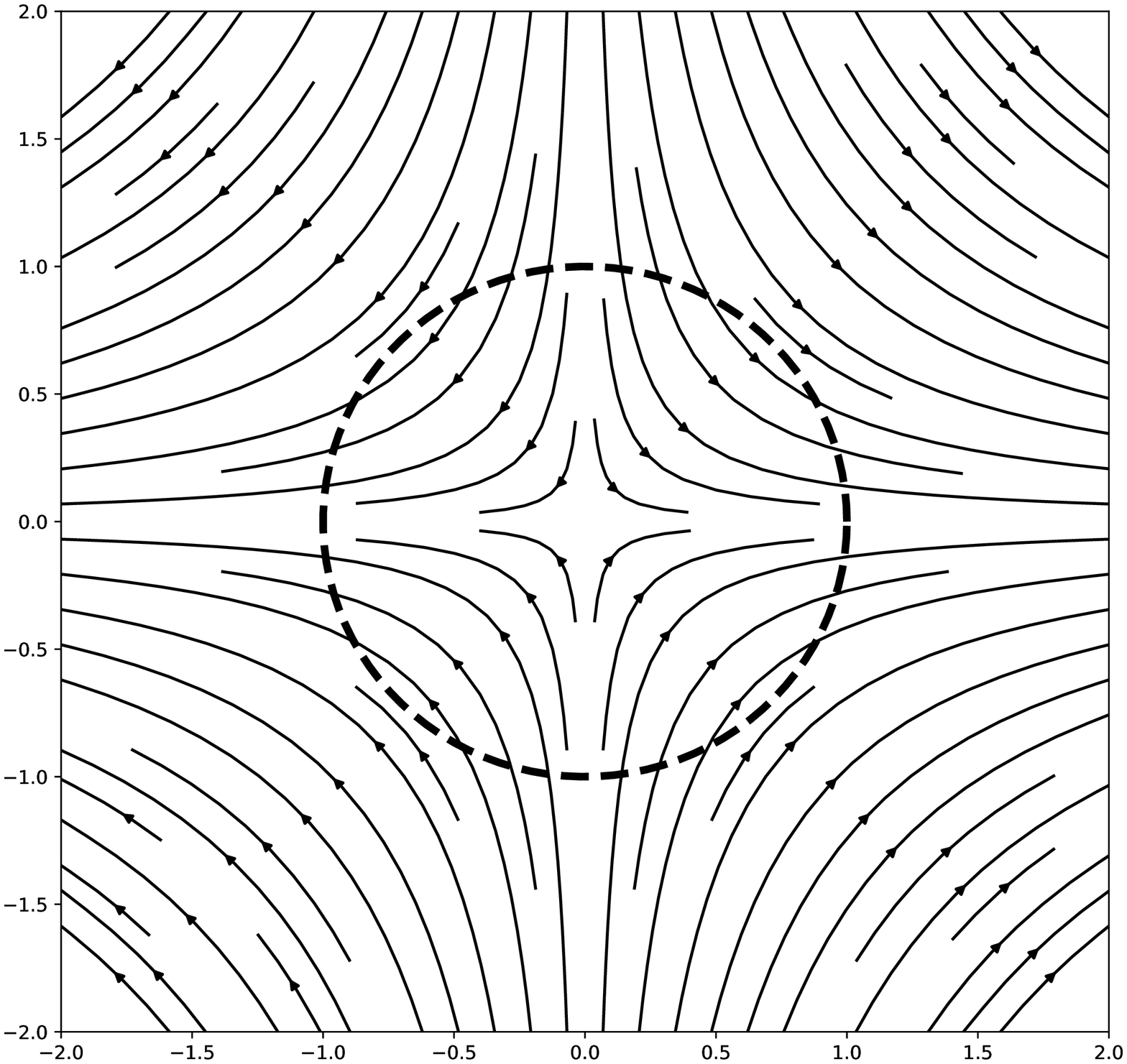}
  \end{minipage}
  \caption{Left: The first-out map and the first-in map of the flow defined by (\ref{ex_exit}). Right: The flow defined by (\ref{ex_exit}).}\label{fig_ex}
\end{figure} 
\end{example}
First-out or first-in maps can be used to describe the transition of states under observation errors, as in the next example.
\begin{example}
 Here we consider the problem of the cooling of an object with a limited supply of heat. If we heat water in a cup by putting a heated stone into it, the temperature of the water will go up and return to room temperature after a sufficiently long time. Let us consider this situation a system of ordinary differential equations for definiteness. Let $T_{s}, T_{w}, T_{r}$ be the temperature of the stone, water, and room, and assume the cooling, or the transfer of heat, is described by Newton's law of cooling, that is,
\[
	\begin{aligned}
		\dv{T_s}{t} & = - \alpha \qty(T_s - T_w)\\
		\dv{T_w}{t} & = \gamma \qty(T_s - T_w) - \beta \qty(T_w - T_r),
	\end{aligned}
\]
where $\alpha, \beta, \gamma, T_r > 0$ are assumed to be constant. The initial condition $T_s(0) = T_H > T_r$ and $T_w(0) = T_r$ is appropriate to describe the situation under consideration. However, we immediately see that $T_w$ will never equal $T_r$, as it never reaches the equilibrium point $T_w = T_s = T_r$ within finite time. Here we must take the error of observation $\epsilon$ into consideration. If we set $A_\epsilon =\{(T_w, T_s) \mid \qty|T_w-T_r|\leq \epsilon\}$, the return to the room temperature can be identified with the return to $A_\epsilon$ after the first exit from it. For example, the total time required to cool down is given in terms of $T^r_{A_\epsilon}$. 
\end{example}

A substantial restriction exists on the possible form of a first-out map since the orbits of a flow are disjoint.
\begin{theorem}\label{thm_two_one}
Let $(X, \Phi)$ be a flow and $A \subset X$ an open set. Then, the first-out map $E_A$ is at most two-to-one.
\end{theorem}
\begin{proof}
First, we show that $E_A(x_1) = E_A(x_2)$ and $x_1 \neq x_2$ imply that $x_1$ or $x_2$ is a fixed point of $E_A$. To obtain a contradiction, we assume
\[
	\begin{aligned}
		T_1 &:= T^e_A\qty(x_1) >0\\
		T_2 &:= T^e_A\qty(x_2) >0.
	\end{aligned}
\]
By the assumption and the property of the flow, we have $x_1 = \Phi\qty(-T_1, E_A(x_1)) = \Phi(T_2-T_1, x_2).$ Since $x_1\neq x_2,$ we have $T_1 \neq T_2$.
Without loss of generality, we may assume that $T_1 < T_2$. As $0\leq T_2 - T_1 < T_2,$, we have $x_1 = \Phi(T_2-T_1, x_2) \in A$, which is a contradiction. Therefore, $T_1= 0$ or $T_2= 0$, which implies $x_1$ or $x_2$ is a fixed point of $E_A$. 

Consequently, if $y\in \partial A$ and $E_A^{-1}(y)$ have two different elements $x_1$ and $x_2$, either of them is $y$. Therefore, the number of elements in $E_A^{-1}(y)$ cannot exceed two.
\end{proof} 
\begin{remark}
The first-out map can be one-to-one when $ A$ is backward invariant. Here we say a subset $A$ to be backward invariant if $\mathcal{O}^-(x) \subset A$ for all $x \in A$.
\end{remark}
Although the first-out and first-in maps are only partially defined, we have the following result. Recall that an open set $A$ is \emph{regular} if $\partial A = \partial \bar A$.
\begin{lemma}\label{lem_class}
Let $(X, \Phi)$ be a flow and $A \subset X$ a regular open set. Then, $E_A(x) = x$ or $R_{\bar A}(x) =x$ for each $x \in \partial A$. Consequently, $\partial A = \mathrm{dom}\, E_A \cup \mathrm{dom}\, R_{\bar A}$. 
\end{lemma}
\begin{proof}
Let $x \in \partial A$ and 
\[
	T:= T^e_A(x).
\]
If $T = 0$, we have $E_A(x) = x$. If $T>0$, then 
\[
	\Phi(t,x) \in A
\]
for all $t \in (0,T)$. Therefore, $R_{\bar A} (x) = x$.
\end{proof}

\subsection{Types of boundary points}
In what follows, we assume that the open set $A$ is always regular.

According to Lemma \ref{lem_class}, each $x\in \partial A$ can be classified into one of the following types.
\begin{enumerate}
	\item (Type A) $x \in  \mathrm{dom}\, E_A \cap \mathrm{dom}\, R_{\bar A}.$
	\item (Type B, never-to-return points) $x \in  \mathrm{dom}\, E_A \backslash  \mathrm{dom}\, R_{\bar A}.$ In this case, $E_A(x) = x$ and $\mathcal{O}^+(x) \subset X \backslash A.$
	\item (Type C, never-to-leave points) $x \in  \mathrm{dom}\, R_{\bar A} \backslash  \mathrm{dom}\, E_A .$ In this case, $R_{\bar A}(x) = x$ and $\mathcal{O}^+(x) \subset \bar A.$
\end{enumerate}
Further, type A can be divided into three subclasses.
\begin{enumerate}
	\item (Type A-1, launching points) $E_A(x) = x$ and $R_{\bar A} (x) \neq x.$
	\item (Type A-2, diving points) $E_A(x) \neq x$ and $R_{\bar A} (x) = x.$
	\item (Type A-3, tangency points) $E_A(x) = x$ and $R_{\bar A} (x) = x.$
\end{enumerate}
In Figure \ref{fig_type}, we present a sketch of a forward trajectory from each type of boundary point. 
\begin{example}
Here we consider a affine system on $\mathbb{R}^2$ given by
\[
	\dv{t} \mqty(x \\ y) = \mqty(-\lambda & -\mu \\ \mu &-\lambda)\mqty(x \\ y - p),
\]
where $\lambda, \mu >0$ and $p \in \mathbb{R}.$
For $A:=\{(x,y) \mid y < 0\}$, the points on the boundary may change their types depending on the value of $p$. By a direct calculation, we obtain the following classification:
\begin{itemize}
	\item When $p > 0$, $(x, 0)$ is of type A-2 if $x < -\frac{\lambda p}{\mu}$ and type B if $x \geq- \frac{\lambda p}{\mu}$.
	\item When $p = 0$,  $(x, 0)$ is of type A-2 if $x <0$, type A-3 if $x =0$, and type A-1 if $x >0$.
	\item When $p < 0$, $(x, 0)$ is of type C if $x \leq -\frac{\lambda p}{\mu}$ and type A-1 if $x >- \frac{\lambda p}{\mu}$.
\end{itemize}
\end{example}
\begin{figure}[htbp]
    \centering
    \includegraphics[keepaspectratio, width = 10cm]{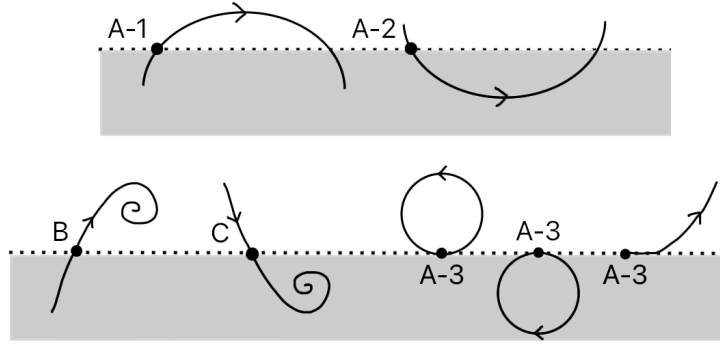}
  \caption{A sketch of a forward trajectory from each type of boundary point. }\label{fig_type}
\end{figure}

These classifications are topological.
\begin{theorem}[Main Theorem A]\label{thm_top}
Let $(X, \Phi)$ and $(Y, \Psi)$ be flows topologically equivalent via a homeomorphism $h: X \to Y$, and $A \subset X$ and $B \subset Y$ be regular open sets with $B = h(A)$. Then, $x \in \partial A$ and $h(x) \in \partial B$ are of the same type. 
\end{theorem}
To prove this theorem, we present a few lemmas.

\begin{lemma}\label{lem_per}
Let $(X, \Phi)$ be a flow and $A \subset X$ an open set. If $x \in \partial A$ is a periodic point with minimal period $T>0$, we have $T^e_A(x) \leq T$ and $T^r_{\bar A}(x) \leq T$.
\end{lemma}

\begin{proof}
The results follow immediately from $\Phi(T, x) = x \in \bar A \backslash A$.
\end{proof}
\begin{lemma}\label{lem_conj}
Let $(X, \Phi)$ and $(Y, \Psi)$ be flows topologically equivalent via a homeomorphism $h: X \to Y$, and $A \subset X$ and $B \subset Y$ be open sets with $B = h(A)$. Then, we have
\[
	\begin{aligned}
		h\left(\mathrm{dom}\, E^\Phi_A\right)  &= \mathrm{dom}\, E^\Psi_B\\
		 h\left(\mathrm{dom}\, P^\Phi_{\bar A}\right)  &= \mathrm{dom}\, P^\Psi_{\bar B}\\
	\end{aligned}
\]
and
\[
	\begin{aligned}
		h\left(E^\Phi_A(x)\right)  &= E^\Psi_B\left(h(x)\right)\\
		 h\left(P^\Phi_{\bar A}(x)\right)  &= P^\Psi_{\bar B}\left(h(x)\right).\\
	\end{aligned}
\]

\end{lemma}
\begin{proof}
It is sufficient to show the results for the first-out maps, as we may use Lemma \ref{lem_sym} to obtain results for the first-in maps. The case for equilibrium points is obvious. Therefore, we may assume that, for each $T$, there exists unique $t >0$ such that $h\qty(\Phi(t,x)) = \Psi\qty(T,h(x))$. This is true even for periodic points because the first exit time is less than the first period by Lemma \ref{lem_per}.

First, we show that $h\left(\mathrm{dom}\, E^\Phi_A\right)  \subset \mathrm{dom}\, E^\Psi_B$. If $x \in \mathrm{dom}\, E^\Phi_A$, there exists $T>0$ with $\Phi(T,x) \not \in A$. Therefore, there exists $T'>0$ with $h\left(\Phi(T,x)\right) = \Psi(T',h(x)) \not \in h(A) = B$. Thus, $h(x) \in  \mathrm{dom}\, E^\Psi_B$. By considering $h^{-1}$, we obtain $h\left(\mathrm{dom}\, E^\Phi_A\right)  = \mathrm{dom}\, E^\Psi_B$.

Now, we show that $h\left(E^\Phi_A(x)\right)  = E^\Psi_B\left(h(x)\right).$ Let 
\[
	\begin{aligned}
		E^\Phi_A(x) &= \Phi(T_1,x),\\
		E^\Psi_B\left(h(x)\right) &= \Psi\left(T_2, h(x)\right),
	\end{aligned}
\]
where
\[
	\begin{aligned}
		T_1 &:= T^e_A\qty(x),\\
		T_2 &:= T^e_B\qty(h(x)).
	\end{aligned}
\]

Since $h\left(E^\Phi_A(x)\right) \in \mathcal{O}^+\left(h(x)\right)$, there exists $T'_1\geq 0$ with $h\left(E^\Phi_A(x)\right)  = \Psi\left(T'_1, h(x)\right).$ By Remark \ref{rem_map}, $h\left(E^\Phi_A(x)\right) \not \in B$. Therefore, $T_2 \leq T'_1$.

To demonstrate a contradiction, we assume $T_2 < T'_1$. In this case, we have
\[h^{-1}\left(E^\Psi_B(h(x))\right) = \Phi(T'_2, x) \not \in A\]
for some $T'_2 \in [0, T_1)$, because $E^\Psi_B(h(x)) \not \in B$. This contradicts the definition of $T_1$. Thus, $T_2 = T'_1$, and therefore, $h\left(E^\Phi_A(x)\right)  = E^\Psi_B\left(h(x)\right)$.
\end{proof}

\begin{proof}[Proof of Theorem \ref{thm_top}]
By Lemma \ref{lem_conj}, types A, B, and C are easily seen to be preserved. The subtypes of type A are also preserved because
we have $E^\Phi_A(x) = x$ if and only if $E^\Psi_B(h(x)) = h(x)$.
\end{proof}
Now we consider the relationship between the types and behavior of orbits. In general, an orbit of a flow may intersect with a boundary of an open set in a complicated fashion. For example, it is possible for a forward orbit from a point $x \in \partial A$ to satisfy the condition that there exist positive sequences $t_n \to 0$ and $s_n \to 0$ such that $\Phi(t_n, x) \in A$ and $\Phi(s_n,x) \not \in \bar A$. This behavior is observable in Example \ref{ex_comsin}, and we may regard it as a kind of complicated tangency. Therefore, as a first step, we would like to restrict our discussion to simpler cases.

In the study of differentiable flows, transversality is a criterion for the behavior of an orbit to be simple. Analogously, here we introduce the following notion of forward topological transversality. 

\begin{definition}\label{Top_trans}
Let $(\Phi, X)$ be a flow, where $X$ is an $n$-dimensional topological manifold. A submanifold $S \subset X$ is \emph{forward topologically transversal} to $\Phi$ at $x\in S$, if
\begin{enumerate}
	\item $S$ is of codimension one and locally flat.
	\item There exists a neighborhood $U$ of $x$ in $X$ and a homeomorphism $\phi: U \to B \subset \mathbb{R}^n$, where $B$ is a unit ball such that $\phi\left(U\cap S \right) = B \cap \mathbb{R}^{n-1}\times \{0\}.$ Further, there exist $\delta_{+}(x) >0$ such that $\Phi((0,\delta_{+}(x)],x)$ is contained in a connected component of $U \backslash S$.\end{enumerate} 

A submanifold $S \subset X$ is \emph{forward topologically transversal} to $\Phi$ if it is forward topologically transversal at every point on $S$. 
\end{definition}
Intuitively, if a point on a submanifold $S$ is forward topologically transversal, then it leaves $S$ and does not return to $S$ for some time. The difference from the usual notion of transversality is that we do not require the orbit to have been somewhere other than $S$ in the past.
\begin{remark}
Here we use the term topological manifold or submanifold under the assumption that they are without a boundary, according to the usage in literature \cite{lee2010introduction}. This is a prerequisite for condition (2) to be valid.
\end{remark}
\begin{remark}\label{rem_top_trans}
Let a submanifold $S \subset X$ is topologically transversal to $\Phi$ at $x\in S$ as defined in \cite{suda2022}, that is,
\begin{enumerate}
	\item $S$ is codimension one and locally flat.
	\item For each $x \in S,$ there exists a neighborhood $U$ of $x$ in $X$ and a homeomorphism $\phi: U \to B \subset \mathbb{R}^n,$ where $B$ is the unit ball such that $\phi\left(U\cap S \right) = B \cap \mathbb{R}^{n-1}\times \{0\}.$ Further, there exist $\delta_{+}(x) >0$ and $\delta_{-}(x) <0$ such that $\Phi(x, [\delta_{-}(x),0))$ and $\Phi(x, (0,\delta_{+}(x)])$ are contained in different connected components of $U \backslash S$ and \[\Phi(x, [\delta_{-}(x),\delta_{+}(x)])) \cap S = \{x\}.\]
	Here, $\delta_{+}$ and $\delta_{-}$ can be taken locally uniformly, namely, there exists a neighborhood $V \subset U$ of $x$ and $\delta >0$ such that $\delta_+(y) > \delta$ and $\delta_-(y) < - \delta$ for all $y \in V\cap S.$
	\item For each set of the form $\Phi(y, [a,b]),$ where $y \in X$ and $a, b \in \mathbb{R},$ $\Phi(y, [a,b])\cap S$ is compact in $S$.
\end{enumerate} 
Then, $\Phi$ and its time reversal $\Psi$ are forward topologically transversal.
\end{remark}
\begin{remark}\label{rem_eq}
If a submanifold $S \subset X$ is forward topologically transversal to $\Phi$ at $x\in S$, $x$ is not an equilibrium point. 
\end{remark}
The difference between forward topological transversality and topological transversality can be observed in the next example.
\begin{example}
We consider the following map from $\mathbb{R} \times \mathbb{R}^2$ to $ \mathbb{R}^2$:
\[
	\Phi(t, x, y) := \begin{cases}
					(x+t, y-t) & (x\geq 0 \text{ and }x+t \geq 0)\\
					(x+t, y-x-t) & (x<0 \text{ and }x+t \geq 0)\\
					(x+t,x+y) & (x\geq 0 \text{ and }x+t < 0)\\
					(x+t, y) & (x<0 \text{ and }x+t < 0)\
				\end{cases}
\]
This map is a continuous flow. We can check that the homeomorphism $h: \mathbb{R}^2 \to  \mathbb{R}^2 $ defined by
\[
		h(x, y) := \begin{cases}
					(x, y) & (x < 0)\\
					(x, x+y) & (x \geq 0). 
				\end{cases}
\]
satisfies $ \Psi(t, h(x,y))=  h \circ \Phi (t, x, y)$, where the flow $\Psi$ is defined by $\Psi(t, x, y) = (x +t ,y)$.
Then, the plane $y =0$ is forward topologically transversal to $\Phi$ at the origin because the orbit is given by
\[
	\qty(x(t), y(t)) = \begin{cases}
					(t, -t) & (t \geq 0)\\
					(t, 0) &(t < 0).
				\end{cases}
\]
 It is not topologically transversal because the backward orbit remain on $y=0$.
\end{example}

Forward topological transversality restricts the possible behavior of orbits. Namely, the forward orbit locally remains in the open set or in the interior of the complement.
\begin{lemma}\label{lem_inside}
Let $X$ be a topological manifold, $\Phi$ be a flow on $X$, $A \subset X$ an open set with the boundary being a locally flat manifold of codimension one. If $\partial A$ is forward topologically transversal to $\Phi$ at $x \in \partial A$, there exists $\delta >0$ with
\[
	\Phi\qty((0, \delta), x) \subset A
\]
 or
 \[
	\Phi\qty((0, \delta), x) \subset \mathrm{int}\,(X\backslash A).
\] 
\end{lemma}
\begin{proof}
Let $x \in \partial A$ be fixed, and U be an open neighborhood of $x$ in the definition of forward topological transversality. First, we observe 
\[
	U\backslash \partial A = (U\cap A) \cup (U\backslash \bar A),
\]
because $x\in \partial A,$ $U\cap A \neq \emptyset$ and $U\backslash \bar A \neq \emptyset$. Therefore, a connected component of $U \backslash \partial A$ is contained in either $U\cap A$ or $U\backslash \bar A$. The conclusion follows from the inclusion $U\backslash \bar A \subset \mathrm{int}\,(X\backslash A)$ and the definition of forward topological transversality.
\end{proof}
There are two possibilities of the behavior in Lemma \ref{lem_inside}. If the types of boundary points are known, we can determine which are feasible.
\begin{theorem}\label{thm_type_class}
Let $X$ be a topological manifold, $\Phi$ be a flow on $X$, and $A \subset X$ an open set with the boundary being a locally flat manifold of codimension one. If $\partial A$ is forward topologically transversal to $\Phi$, we have the following: 
\begin{enumerate}
	\item If $x \in \partial A$ is of type A-2 or C, there exists $\delta >0$ with
\[
	\Phi\qty((0, \delta), x) \subset A.
\]
	\item If $x \in \partial A$ is of type A-1 or B, there exists $\delta >0$ with
 \[
	\Phi\qty((0, \delta), x) \subset \mathrm{int}\,(X\backslash A).
\]
\end{enumerate}
\end{theorem}
\begin{proof}
We show the statement (1) because the proof of (2) is similar. Let $x \in \partial A$ be type A-2 or C.
By Lemma \ref{lem_inside}, there are two possibilities regarding the behavior of the forward orbit of $x$. To prove by contradiction, we assume that 
 \[
	\Phi(t, x) \in \mathrm{int}\,(X\backslash A)
\]
for all $t \in (0, \delta).$ Then, we have $E_A(x) = x$, which is not consistent with type A-2 or C. Therefore, we have
\[
	\Phi(t, x) \in A
\]
for all $t \in (0, \delta)$. 
\end{proof}
Type A-3 may be regarded as a degenerate case. The following theorem shows that other types imply forward topological transversality.
\begin{theorem}\label{thm_type}
Let $X$ be a topological manifold, $\Phi$ be a flow on $X$, and $A \subset X$ an open set with the boundary being a locally flat manifold of codimension one. If $\partial A$ is not forward topologically transversal to $\Phi$ at $x \in \partial A$,  $x$ is of type A-3.
\end{theorem}
\begin{proof}
 If $\Phi$ is not forward topologically transversal to $\partial A$ at $x \in \partial A$, we may find a sequence $t_n >0$ with $\Phi(t_n, x) \in \partial A$ and $t_n \to 0$ as $n \to \infty.$ Then, $E_A(x) = R_{\bar A}(x) = x$ by definition.
\end{proof}
If type A-3 occurs at a boundary point with forward topological transversality, it should be a part of a periodic orbit. This is an analog of the classical result that the fixed points of first-return maps correspond to periodic points of the original flow.
\begin{theorem}\label{thm_a3}
Let $X$ be a topological manifold, $\Phi$ be a flow on $X$, and $A \subset X$ an open set with the boundary being a locally flat manifold of codimension one. If $\partial A$ is forward topologically transversal to $\Phi$ at $x \in \partial A$, $x$ is of type A-3 if and only if $x$ is a periodic point with
\[
	\mathcal{O}(x)\backslash \{x\} \subset A
\]
or
\[
	\mathcal{O}(x)\backslash \{x\} \subset X \backslash \bar A.
\]
\end{theorem}
\begin{proof}
Let $x \in \partial A$ be of type A-3 and
\[
	\begin{aligned}
		T_1 &:= T^e_A(x)\\
		T_2 &:= T^r_{\bar A}(x).
	\end{aligned}
\]
By Lemma \ref{lem_inside}, $T_1=0$ and $T_2 >0$ or $T_1 > 0$ and $T_2 =0$. We now consider the former case.

As in Remark \ref{rem_eq}, $x$ is not an equilibrium point. Therefore, $x = R_{\bar A}(x) = \Phi(T_2, x)$ is a periodic point. By the definition of $T_2$, we have $\Phi(t,x)\not \in \bar A$ for all $t \in (0, T_2)$. Thus, we obtain $\mathcal{O}(x)\backslash \{x\} \subset X \backslash \bar A$. The proof for the other case is similar.

Conversely, let $x$ be a periodic point with $\mathcal{O}(x)\backslash \{x\} \subset A $,
and the minimal period be $T>0$. Then, we have $\Phi(t,x) \in A$ for all $t \in (0,T)$, and consequently, $R_{\bar A}(x) = x$. Because $\Phi(T,x) \not \in A,$ $E_A(x) =\Phi(T,x)= x.$ Therefore, $x \in \partial A$ is of type A-3. The proof for the case $\mathcal{O}(x)\backslash \{x\} \subset X \backslash \bar A$ is similar.
\end{proof}

The invariance of open subsets can be expressed in terms of the type of points on the boundary. This can be seen as a generalization of similar results regarding smooth manifolds and smooth flows.

\begin{theorem}\label{thm_bkwd}
Let $(X, \Phi)$ be a flow and $A \subset X$ an open set. Then, $A$ is backward invariant if and only if all points on $\partial A$ are of type B.
\end{theorem}
\begin{proof}
 Let $x \in \partial A$. If $A$ is backward invariant, then $X \backslash A$ is forward invariant. Therefore, $\mathcal{O}^+(x) \subset X \backslash A$ and $E_A(x) =x$. Thus, $x$ is of type B. 

Conversely, we assume all points on $\partial A$ are of type B. If there exists $x \in X \backslash A$ with $\Phi(T,x) \in A$ for some $T>0$, there is $t_0 \in (0, T)$ such that $y :=\Phi(t_0, x) \in \partial A$ by Lemma \ref{lem_IVT}. This is contradictory because $y$ is of type B, and consequently, $\Phi(T,x) = \Phi(T-t_0, y) \not \in A$. Therefore, $A$ is backward invariant.
\end{proof}

Due to the problem of tangency, a similar characterization of forward invariance is more complicated. 

\begin{theorem}
Let $X$ be a topological manifold, $\Phi$ be a flow on $X$, and $A \subset X$ be an open set with the boundary being a locally flat manifold of codimension one. Then, $A$ is forward invariant and $\partial A$ is forward topologically transversal to $\Phi$ if and only if all points on $\partial A$ are of type C.
\end{theorem}
\begin{proof}
Let $A$ be forward invariant, and $\partial A$ be forward topologically transversal to $\Phi$. We fix $x \in \partial A$. By Lemma \ref{lem_inside}, there are two possible cases for the behavior of $x$. Since $\bar A$ is forward invariant, there exists $\delta >0 $ with $\Phi((0, \delta), x) \subset A$. By the invariance of $A$, it follows that $\Phi(t,x) \in A$ for all $t>0$. Therefore, $x$ is of type C.

Conversely, let all points on $\partial A$ be of type C. By Theorem \ref{thm_type}, $\Phi$ is forward topologically transversal to $\partial A$. If there exists $x \in A$ with $\Phi(T, x) \not \in A$ for some $T>0$, there is $t_0 \in (0, T]$ with $y:= \Phi(t_0, x) \in \partial A$ by Lemma \ref{lem_IVT}. Since $y$ is of type C, $\Phi(T, x) = \Phi(T-t_0, y) \in A$. This is a contradiction, and therefore, $A$ is forward invariant.
\end{proof}
Since the first-out map is defined only partially, it is of interest to consider the topological properties of the domain. If the assumption of forward topological transversality is imposed, we may obtain some information regarding this point. 

\begin{theorem}
Let $X$ be a topological manifold, $\Phi$ be a flow on $X$, $A \subset X$ an open set with the boundary being a locally flat manifold of codimension one. If each point in $\mathrm{im}\, E_A$ is of type A-1 or B and $\partial A$ is forward topologically transversal to $\Phi$, the domain of $E_A$ is open in $\partial A$. 
\end{theorem}
\begin{proof}
Let $x\in \mathrm{dom}\, E_A$ and 
\[
	T: = T^e_A(x).
\] 
Let us first consider the case where $T=0.$ Then, by the hypothesis, $x = E_A(x)$ is of type A-1 or B. By Theorem \ref{thm_type_class}, there exists $\tau>0$ with
\[
	\Phi(\tau, x) \in \mathrm{int}\, (X \backslash A).
\]
By the continuity of $\Phi,$ we may take an open neighborhood $U$ of $x$ in $X$ such that $\Phi(\tau, U) \subset X \backslash A.$ Therefore, $\partial A \cap U \subset \mathrm{dom}\, E_A.$

Next, let us consider the case where $T>0.$ Then, by the hypothesis, $y = E_A(x) = \Phi(T, x)$ is of type A-1 or B. By the aforementioned argument, there is an open neighborhood $V$ of $y$ and $\tau >0$ such that $\Phi(\tau ,V) \subset  X \backslash A.$ Therefore, $\partial A \cap \Phi(-T, V) \subset \mathrm{dom}\, E_A.$
\end{proof}
\begin{corollary}
Let $X$ be a topological manifold, $\Phi$ be a flow on $X$, $A \subset X$ an open set with the boundary being a locally flat manifold of codimension one. If $E_A$ is idempotent, $\partial A$ is forward topologically transversal to $\Phi$, and there are no periodic points on $\partial A$, then the domain of $E_A$ is open in $\partial A$. 
\end{corollary}
\begin{proof}
Because $E_A$ is idempotent, each point in $\mathrm{im}\, E_A$ is of type A-1, A-3, or B. By Theorem \ref{thm_a3}, the type A-3 is incompatible with the assumption that there are no periodic points. Therefore, each point in $\mathrm{im}\, E_A$ is of type A-1 or B.
\end{proof}

If we further impose the assumption of topological transversality, the continuity of first-out maps can be shown.

\begin{theorem}
Let $X$ be a topological manifold, $\Phi$ be a flow on $X$, $A \subset X$ an open set with the boundary being a locally flat manifold of codimension one. If $\partial A$ is topologically transversal to $\Phi$ in the sense of the definition in Remark \ref{rem_top_trans} at $x$ and $E_A(x)$, then, $E_A$ is continuous at $x$. 
\end{theorem}
\begin{proof}
The continuity follows from Main Theorem B in \cite{suda2022} since $\partial A$ is a local section of $\Phi$ at $x$ and $E_A(x)$. 
\end{proof}
%
%
\section{First-out maps of planar flows}
This section considers the restriction for the first-out maps and first-in maps of planar flows. It is natural to expect that only some partial maps can be derived as a first-out map for some flow because correspondences are restricted by the property that orbits of a flow never intersect each other. This restriction can be analyzed in a rather concrete form in the planar case.

If the boundary of an open set is a Jordan curve, each parametrization induces a sequence of types. As there are forbidden combinations of types, the possible forms of first-out maps and first-in maps can be restricted. Moreover, this sequence of types can be used to study the dynamics around boundary points because they reflect the local dynamics. 

Another way to consider restrictions is the parametrized representation of first-out maps. In this case, they are just one-dimensional partial maps. Here we will consider the necessary conditions for a partial map to be derived from a parametrization of a first-out map for some flow.
\subsection{Type sequence}

If the boundary of an open subset of $\mathbb{R}^2$ is parametrized, a sequence of types is naturally defined. Note that an open set encircled by a Jordan curve is regular by the Jordan--Schoenflies theorem.
\begin{definition}
Let $\Phi: \mathbb{R}\times \mathbb{R}^2 \to \mathbb{R}^2$ be a continuous flow and $A \subset \mathbb{R}^2$ be an open subset with $\partial A$ being a Jordan curve. 
For a parametrization $c:[0,1]$ of $\partial A$, the \emph{type sequence} of $c$ is a map $L_c:[0,1] \to \{\text{A-1}, \text{A-2}, \text{A-3},\text{B}, \text{C}\}$ defined by setting $L_c(t)$ to be of type $c(t)$.
\end{definition}
\begin{example}
Let us consider the flow in Example \ref{ex_exmap}. If we parametrize the unit circle by $c:[0,1) \to \mathbb{R}^2$ where $c(t):=\qty(\cos 2 \pi t, \sin 2 \pi t),$ we obtain
\[
	L_c(t) =\begin{cases}
				\text{B} & (0\leq t \leq \frac{1}{8})\\
				\text{A-2} &( \frac{1}{8} < t < \frac{1}{4})\\
				\text{C} & (t = \frac{1}{4})\\
				\text{A-2} &( \frac{1}{4} < t < \frac{3}{8})\\
				\text{B} & (\frac{3}{8}\leq t \leq \frac{5}{8})\\
				\text{A-2} &( \frac{5}{8} < t < \frac{3}{4})\\
				\text{C} & (t = \frac{3}{4})\\
				\text{A-2} &( \frac{3}{4} < t < \frac{7}{8})\\
				\text{B} & (\frac{7}{8} \leq t <1)\\
			\end{cases}
\]
\end{example}
%
First, we note that there is a forbidden combination of types. In what follows, $A$ is an open set with $\partial A$ being a Jordan curve. For a parametrization $c$ of $\partial A$, we say that an interval $I \subset [0,1]$ \emph{comprises} a type $\alpha$ if each point $c(t)$, where $t\in I$, has type $\alpha$.  

\begin{theorem}
The combination of types B and C does not occur in any type sequence.
\end{theorem}
\begin{proof}
 The combination BC is impossible by the following argument. Let $c$ be a parametrization of $\partial A.$ If $I = (p, q)$ comprises type B, it can be shown that $\mathcal{O}^+(c(q)) \subset X \backslash A$. Similarly, if $J = (q, r)$ comprises type C, $\mathcal{O}^+(c(q)) \subset \bar A$. Therefore, if two open intervals of type B and C are juxtaposed , the common point $q$ of their closures is of type A-3.
\end{proof}


The dynamics around the junction of different types can be inferred from the combination of types.

\begin{theorem}
Let $c$ be a parametrization of $\partial A.$ If $I =(p, q)$ comprises type A-1 and $J =[q , r)$ comprises type A-2 or C, there exists $\tau > 0$ such that $\Phi([-\tau, 0) \cup (0, \tau], c(q)) \subset A$.
\end{theorem}
\begin{proof}
First, we show that $R_{\bar A}(c(x)) \not \in c \qty(I)$ for all $x \in I$. To obtain a contradiction, we assume $R_{\bar A}(c(x)) = c(s)$ with $s \in I$. By the definition of the first-in map and the assumption of A-1, each sufficiently small $\tau>0$ has a corresponding neighborhood $V_{\tau}$ of $c(s)$ such that
\[
	\Phi(-\tau, V_{\tau}) \subset X \backslash \bar A.
\]
We now consider a sequence $\tau_n$ with $\tau_n \to 0$ as $n \to \infty$. Then, we may define sequences $\{z_n\}$ with $z_n \in V_{\tau_n} \cap A$ and $\{\sigma_{n}\}$ such that 
\[
		\sigma_n = \inf\{\sigma>0 \mid \Phi\qty(-\sigma, z_n) \not \in A\}.
\]
Note that $\sigma_n >0$, because $A$ is open and $z_n \in A$. Since $\Phi\qty(-\tau_n, z_n) \not \in A,$ $\sigma_n\leq \tau_n$. Therefore, $\sigma_n \to 0$ as $n \to \infty$. Moreover, we have 
\[
\begin{aligned}
	\Phi\qty(-\sigma_n, z_n) &\in \partial A\\
	\lim_{n \to \infty} \Phi\qty(-\sigma_n, z_n) &= c(s)
\end{aligned}
\]
by the continuity of $\Phi$. Further,  $\Phi\qty(-\sigma_n, z_n)$ is not type A-1, since we have $R_{\bar A}\qty( \Phi\qty(-\sigma_n, z_n)) = \Phi\qty(-\sigma_n, z_n)$ by the choice of $\sigma_n$. Therefore, $c(s)$ is in the closure of $\partial A \backslash c \qty(I)$ in $\partial A$, which contradicts the assumption $c(s) \in c \qty(I)$.

Next, we show that for all $\delta >0$ and an interval $U = (a, q)$, we have 
\[
T^r_{\bar A}\qty(c(x))\leq \delta\\
\]
for some $x \in U$. 
We assume that there exists $\delta >0$ and an interval $U = (a, q)$ such that 
\[
T^r_{\bar A}\qty(c(x)) > \delta\\
\]
for all $x \in U$. 

 Then, by the continuity of $\Phi,$ we have
\[
	\Phi(t, c(q)) = \lim_{n \to \infty} \Phi\left(t, c\left(q-\frac{1}{n}\right)\right) \in \overline{X \backslash \bar A} \subset X \backslash A,
\]
for all $t \in (0, \delta)$. Therefore, $q$ cannot be of type A-2 or C.

Hence, we may construct sequences $x_n$ and $t_n>0$ such that $x_n \to q,$ $t_n \to 0$ and
\[
	R_{\bar A}(c(x_n))= \Phi(t_n,c(x_n)) \in \partial A.
\]
We find $s_n \in (0, t_n)$ with $ \Phi(s_n,c(x_n)) \not\in  A$.

Let $W$ be a sufficiently small neighborhood of $c(q)$ such that $W \cap \partial A \subset c([p,r])$. By the continuity of $\Phi$, there exists $\tau_0 >0$ and a neighborhood $W'$ of $c(q)$ such that
\[
	\Phi\qty((-\tau_0,\tau_0)\times W') \subset W.
\]
Therefore, for sufficiently large $n$, we have $R_{\bar A}(c(x_n))= \Phi(t_n,c(x_n)) \in W \cap \partial A$. Moreover, since $R_{\bar A}(c(x_n)) \not \in c(I)$, we have
\[
	R_{\bar A}(c(x_n)) \in c(J).
\]
By considering the orbits, it follows that $R_{\bar A}(c[x_n,q)) \subset c(J)$. Let $W''$ be the open set encircled by the forward orbit of $c(x_n)$ and $\partial A$.

We observe that $\mathcal{O}^{-} (c(q)) \cap W'' = \emptyset$ because otherwise, we have an intersection of orbits. Let us consider another sufficiently small neighborhood $W'''$ of $c(q)$ such that $W''' \backslash \bar A \subset W''.$ For each sufficiently small $\tau' >0$, we have $\Phi\qty(-\tau', c(q))\in W'''$ by continuity, and consequently, $\Phi\qty(-\tau', c(q))\in \bar A$. Since there is no point of type A-3 around $c(q)$, it follows that $\Phi( [-\tau, 0), c(q)) \subset A$ for some $\tau >0.$
\end{proof}

We remark that type sequences define a topological invariant. Namely, if two flows are topologically equivalent, then the type sequence is the same for two open sets that correspond under the homeomorphism of topological equivalence. Therefore, two flows cannot be topologically equivalent if there is a combination of types that appears only for one of the two.
\begin{example}[Sink and source are not topologically equivalent]
Let us consider two flows defined by
\[
	\Phi(t,x,y) = (xe^{t}, ye^{t})
\]
and 
 \[
	\Psi(t,x,y) = (xe^{-t}, ye^{-t}).
\]
For $\Phi$, the boundary of the unit disc comprises type B. If $\Phi$ and $\Psi$ are topologically equivalent, then the interior of the unit disk would be mapped to a bounded open set with its boundary being type B. However, for $\Psi$, open sets are unbounded if the boundary comprises type B. This is a consequence of Theorem \ref{thm_bkwd}.
Therefore, $\Phi$ and $\Psi$ cannot be topologically equivalent.
\end{example}

\subsection{Parametric representation of the first-out map}

Let $A$ be an open set with $\partial A$ being a Jordan curve.
If $c$ is a parametrization of $\partial A$, we may define a partial map $F_E: [0,1) \to [0,1)$ by setting \[c(F_E(s)) = E_A(c(s)).\] 
The partial map $F_E$ encodes information of $E_A$. While $F_E$ does not reflect the full information on the type of a boundary point, it is easier to analyze as it is a one-dimensional partial map.
\begin{remark}
$F_E$ is defined at $s$ if and only if $E_A$ is defined at $c(s)$. Further, if $F_E$ is continuous at $s$, $E_A$ is continuous at $c(s)$. The converse is true for $s \in (0,1)$. Moreover, note that $F_E(s) \neq s$ if and only if $c(s)$ is of type A-2.
\end{remark}
Another restriction can be described in terms of $F_E$. Since the orbits of a flow are disjoint, $F_E$ satisfies a monotonicity condition. First, we consider this in terms of the first-out map.

\begin{lemma}\label{lem_mono}
Let $\Phi: \mathbb{R}\times \mathbb{R}^2 \to \mathbb{R}^2$ be a continuous flow, $A \subset \mathbb{R}^2$ be an open subset, with $\partial A$ being a Jordan curve, and $c:[0,1)$ a parametrization of $\partial A$ and $E_A(c(t_0)) = c(t_1)$ with $t_1 > t_0.$ If $E_A(c(s_0)) = c(s_1)$ and $s_0 \in (t_0, t_1)$. Then, we have $s_1 \in (t_0, t_1)$. 
\end{lemma}
\begin{proof}

Let $\gamma_1$ be a Jordan curve defined by 
\[
	\gamma_1(t) = \begin{cases}
				c(2 (t_1-t_0)t +t_0) & (t \in [0, 1/2))\\
				\Phi\qty(2T^e_{A}\qty(c(t_0))(1-t),c(t_0)) & (t \in [1/2,1))
			\end{cases}
\] 
and $A_1$ be the interior of the domain encircled by $\gamma_1$. Then, $A\backslash A_1$ is also encircled by a Jordan curve and $c(s_0) \not \in \overline{A\backslash A_1}$.

We assume $s_1 \not \in (t_0, t_1)$. Since $\Phi(t,c(s_0)) \in A_1$ for some $t \in \qty(0,T^e_{A}\qty(c(s_0)))$, there exists $t' \in (0,T^e_{A}\qty(c(s_0)))$, such that $\Phi(t',c(s_0)) \in \partial A_1$ by Lemma \ref{lem_IVT}. Since $\Phi(t',c(s_0)) \in A$, $z:=\Phi(t',c(s_0))=\Phi(t'', c(t_0))$ for some $t'' \in (0,T^e_{A}\qty(c(t_0)))$. This contradicts $s_0 \neq t_0$, as $\inf\{\tau >0 \mid \Phi(-\tau, z) \not \in A\} = t' = t''$.
\end{proof}
Then, this result can be restated in terms of parametric representation.
\begin{lemma}\label{lem_mono_F}
If $F_E(s) = t$ with $s<t$, $F_E(s') \in (s, t)$ whenever $s' \in (s,t)$. Similarly, if $F_E(s) = t$ with $s>t$, $F_E(s') \in (t, s)$ whenever $s' \in (t, s)$.
\end{lemma}
\begin{proof}
The first statement is a direct consequence of Lemma \ref{lem_mono}. The second statement can be obtained from the consideration of another parametrization
\[
	\tilde c(t) := \begin{cases}
					c(0) & (t=0)\\
					c(1-t) &(0<t<1).
				\end{cases}
\]
\end{proof}

This is a substantial restriction, as we can observe in the following theorem, which is a generalization of the classical result on the monotonicity of one-dimensional first-in maps.
\begin{theorem}[Main Theorem B]\label{thm_mono}
If $F_E$ is continuous at $s$ and $F_E(s)\neq s$, there exists $\delta>0$ such that $F_E(t) > F_E(s)$ whenever $s-\delta < t <s$ and $F_E(t) < F_E(s)$ whenever $s<t<s+\delta$.
\end{theorem}
\begin{proof}
Let us first consider the case where $F_E(s) > s$. By the continuity of $F_E$ at $s$, there exists $\delta > 0$ such that $ \left| F_E(t) - F_E(s)\right|< \epsilon = F_E(s) -s$ if $t \in (s-\delta, s+\delta) \cap \mathrm{dom} \, F_E$. Without loss of generality, we may also assume $\delta <\epsilon$ and $F_E(t) > t$ for all $t \in (s-\delta, s+\delta) \cap \mathrm{dom} \, F_E$. 

If $t \in (s-\delta, s)$, we have $F_E(s) - F_E(t) < \epsilon = F_E(s) -s$, and therefore, $ s \in (t, F_E(t))$. By Lemma \ref{lem_mono_F}, we have $F_E(s) < F_E(t)$.

If $t \in (s, s+\delta)$, we have $t < s + \delta < s + \epsilon = F_E(s)$. Therefore, $t\in (s, F_E(s))$, which implies $F_E(t) < F_E(s)$. 

The proof for $F_E(s) < s$ is similar.
\end{proof}
\begin{corollary}
If $F_E$ is defined and monotonically increases on an interval $I = (a,b)$, then $F_E$ is the identity on $I$, except at most countable points.
\end{corollary}
\begin{proof}
By monotonicity, $F_E$ is continuous on $I$, except at most countable points. By applying Theorem \ref{thm_mono} to continuous points, we obtain this result.
\end{proof}
\begin{corollary}\label{cor_min}
If $F_E$ takes a local maximum (minimum) at $s\neq 0$, then $F_E$ is either discontinuous at $s$ or $F_E(s) =s$.
\end{corollary}
\begin{proof}
If $F_E$ is locally maximal (minimal) at $s$ and continuous at $s$, Theorem \ref{thm_mono} implies that $F_E(s) =s$.
\end{proof}

\begin{corollary}\label{cor_con}
If $F_E$ is defined and continuous on $[0,1)$, then the number of minimum or maximum of $F_E$ is at most one for each. 
\end{corollary}
\begin{proof}
First, we show that if there exists $t \in [0,1)$ with $F_E(t) < t$, then $F_E(s) < F_E(t)$ for all $s \in (t, 1)$. We assume that there exists $s \in (t,1)$ with $F_E(s) \geq F_E(t)$. By applying Theorem \ref{thm_mono} at $t$, we see that $F_E$ takes a minimum on $[t,s]$ at some $t' \in (t,s)$. By Corollary \ref{cor_min}, $F_E(t') = t'$. Thus, we obtain
\[
	t < t' = F_E(t') \leq F_E(t),
\]
which contradicts the assumption of $t$. 

By a similar argument, we have $F_E(s') < F_E(s)$ for all $s \in (t, 1)$ and $s' \in (s, 1)$ if $F_E(t) < t$. Therefore, $F_E$ decreases monotonically on $[t,1)$ if $F_E(t) < t$. Similarly, $F_E$ decreases monotonically on $[0,t]$ if $F_E(t) > t$.

Thus, if we set 
\[
	\begin{aligned}
		\alpha& := \sup\{t\in [0,1) \mid F_E(t) > t\},\\
 		\beta&:=\inf\{t\in [0,1) \mid F_E(t) < t\},
	\end{aligned}
\]
then $F_E$ monotonically decreases on $[0, \alpha)$ and $(\beta, 1)$. Here we note that $F_E(s) > s$ for all $s \in [0, \alpha)$ by monotonicity. Thus we have $\alpha \leq \beta$. As $F_E$ is identity on $[\alpha, \beta]$, it monotonically increases on $[\alpha, \beta]$.
\end{proof}
\begin{remark}\label{rem_uni}
If $F_E$ is continuous on $[0,1)$ and not equal to the identity, $F_E$ can be modified to be unimodal. Let $\alpha$ and $\beta$ be as in Corollary \ref{cor_con}. We define another parametrization $\tilde c$ by
\[
	\tilde c (t) := c(t + \alpha).
\]
Then, we may define a unimodal function by
\[
	\tilde F_E(t) = \begin{cases}
						F_E(t+\alpha) -\alpha & (t\in [0, 1-\alpha))\\
						F_E(t+\alpha-1)-\alpha & (t\in [1-\alpha, 1)).
					\end{cases}
\]
\end{remark}

\subsection{Realization of first-out maps}

We now consider a converse question: given a partial map on the boundary, can we find a flow such that the first-out map coincides with it? As we will see later in Section 5, the answer to this question is relevant to the analysis of hybrid systems, where the dynamics are described using both flows and maps. 

This problem is generally solvable for globally-defined continuous first-out maps.
\begin{theorem}[Main Theorem C]\label{thm_conti}
Let $P: \mathbb{R} \to \mathbb{R}$ be a continuous map such that $P(-\infty, 0] = [0, \infty)$, $P(0) =0$, $P$ is two-to-one, except at $0$ and identity on $[0, \infty)$. Then, $(P, 0) = E_{H^{-}}^\Phi$ for some flow $\Phi,$ where $H^{-} :=\{(x,y) \mid y <0\}$.
\end{theorem}
\begin{proof}
By the assumption of two-to-oneness, $P|_{(-\infty, 0]}:(-\infty, 0] \to [0, \infty)$ is a continuous bijection and, therefore, decreases strictly. We set up a continuous map $F^{-}: [0,1] \times (-\infty, 0)  \to \mathbb{R}^2$ by
\[
	F^-(t, x) : = \left( \begin{array}{c}
					-R^-(t,x)\cos \pi t,\\
					-R^-(t,x)\sin \pi t
				\end{array}
				\right),
\]
where $R^-(t,x):= t P(x) -(1-t) x$. Similarly, we define another continuous map $F^{+}: [0,1] \times (0, \infty)  \to \mathbb{R}^2$ by
\[
	F^+(t, x) : = \left( \begin{array}{c}
					R^+(t,x)\cos \pi t,\\
					R^+(t,x)\sin \pi t
				\end{array}
				\right),
\]
where $R^+(t,x):= (1-t)  x -t P^{-1}(x)$, and $P^{-1}$ denotes the negative branch.  

It can be shown by a direct calculation that $F^{-}: [0,1] \times (-\infty, 0)  \to \overline{H^-}\backslash \{(0,0)\}$ is a continuous bijection. Since inverse images of bounded sets are bounded, $F^{-}$ is a proper map. Therefore, $F^{-}$ is a homeomorphism. Similarly, $F^+$ is also a homeomorphism to its image. Now, we define a homeomorphism $H: \mathbb{R}^2 \to \mathbb{R}^2$ given in polar coordinates by
\[
	H(r, \theta):=\begin{cases}
				(0,0) &(r=0)\\
				F^-\qty(1+\theta/\pi, -r) &( \theta \in[-\pi, 0))\\
				F^+\qty(\theta/\pi, P(-r))&( \theta \in [0, \pi]),
				\end{cases}
\]
and a flow $\Psi$ by $\Psi(t, x,y):=\left(\cos(\pi t) x -\sin(\pi t)y, \sin(\pi t) x + \cos(\pi t) y\right)$. Note that $H$ maps the circle $r = r_0>0$ to a closed curve as
\[
	\gamma_{r_0}(\theta) := \begin{cases}
						R^{-}\qty(1+ \theta/\pi, -r_0) & (\theta \in [-\pi, 0))\\
						R^{+}\qty(\theta/\pi, P(-r_0)) & (\theta \in [0, \pi]).
					\end{cases}
\]
Then, it can be checked that $\Phi(t, x, y):= H\left(\Psi(t, H^{-1}(x,y)) \right)$ is a flow with the desired properties.
\end{proof}
\begin{corollary}\label{cor_circ}
Let $P: [0,1) \to [0,1)$ be a continuous unimodal map with $\lim_{t \to 1} P(t) =0$, $P(\alpha)=0$, and $P$ be identity on $[0, \alpha]$, where $P$ takes the maximum at $\alpha$. Then, $P = F_E$ for some flow $\Phi$ and $A = D^2$.
\end{corollary}
\begin{proof}
First, we consider the case where $\alpha = \frac{1}{2}$.
By the hypothesis, $P$ induces a map $\hat{P}: S^1 \to S^1$ by $\hat{P}(e^{2 \pi i t} )= e^{2 \pi i P(t)}$.
Let $M$ be an M\"obius transformation such that $M(-1) = 0,$ $M(1) = \infty$, mapping the open unit disc to the lower half plane. For example, $M(z) = i\frac{z+1}{z-1}$ satisfies the conditions. Then, it follows that the map $Q:\mathbb{R} \to \mathbb{R}$ defined by $Q(x) = M\left(\hat{P}(M^{-1}(x))\right)$ satisfies the hypotheses of Theorem \ref{thm_conti}. Therefore, we may construct a flow $\Psi$ such that the first-out map is $Q$. As in Example \ref{ex_moe}, we may find another flow $\Phi$ with $M\left(\Phi(t,z)\right) = \Psi(t, M(z))$. By Lemma \ref{lem_conj}, we have
\[
\begin{aligned}
	E_{D^2}(e^{2\pi i t}) &= E_{D^2}\left(M^{-1}\circ M(e^{2\pi i t}) \right) \\
					&= M^{-1}\circ Q\left(M(e^{2\pi i t})\right)\\
					&= \hat{P}(e^{2 \pi i t}) = e^{2 \pi i P(t)},
\end{aligned}
\]
for the first-out map of $\Phi$.

For the case $\alpha \neq \frac{1}{2}$, let $h:[0,1] \to [0,1]$ be a homeomorphism such that $h(0) = 0$ and $h(\alpha) = \frac{1}{2}$. By applying the preceding arguments for $\tilde P:=h \circ P \circ h^{-1}$, we obtain a flow $\tilde \Phi$ on $\mathbb{R}^2$ with $E_{D^2}(e^{2\pi i t}) = e^{2 \pi i \tilde P(t)}.$ If we define $H(r e^{i \theta}) := r e^{2 \pi i h(\theta / 2 \pi)}$ in polar coordinates, the map $H$ is a homeomorphism. If $\Phi$ is the flow conjugate with $\tilde\Phi$ via $H^{-1},$ it can be verified that it is the desired flow.
\end{proof}

So far, we have considered globally defined maps. For general partial maps, a promising approach will be to construct a flow with prescribed type sequences by pasting the flows with known types. However, it requires a consideration of the behavior, and therefore, it needs to be clarified whether it is always feasible.
%
\section{Application to the study of hybrid systems}

As an application of the results obtained earlier, we now consider a class of hybrid systems and consider their relationship with the flows.

A hybrid system consists of flows and maps defined locally, and the notion of partial maps is useful in describing them. First we define the notion of local flow as follows. 
\begin{definition}[local flow]\label{def_loc}
Let $X$ be a topological space.
A partial map $\Phi:\mathbb{R}\times X \to X$ with the open domain is a \emph{local flow} if it satisfies the following conditions.
\begin{enumerate}
	\item For each $x\in X$, there exist $-\infty\leq \alpha_x < 0$ and $0<\beta_x \leq \infty$ such that $(t,x) \in \mathrm{dom}\, \Phi$ if and only if $t\in (\alpha_x,\beta_x)$.
	\item $\Phi(0,x) = x$ for all $x \in X.$
	\item If $(s,x) \in \mathrm{dom}\, \Phi$ and either $(t+s,x) \in \mathrm{dom}\, \Phi$ or $\left(t, \Phi(s, x) \right) \in \mathrm{dom}\, \Phi$, $\Phi(t+s,x)=\Phi\left(t, \Phi(s, x) \right)$.
\end{enumerate}
For a local flow $\Phi$, we define the \emph{orbit} of $x\in X$ by
\[
	\mathcal{O}(x) := \{\Phi(t,x) \mid (t, x) \in \mathrm{dom}\, \Phi\}.
\]
\end{definition}
\begin{remark}
Our notion of local flow differs from the one appearing in literature in that we do not require the domain of each orbit to be well-behaved \cite{sell1971topological}. If the non-extendability condition is imposed, this definition essentially coincides with a \emph{local dynamical system} \cite{ura1969isomorphism, mccann1971}.  
\end{remark}

\begin{lemma}\label{lem_disj}
The orbits of a local flow are disjoint.
\end{lemma}
\begin{proof}
Let $z \in \mathcal{O}(x) \cap \mathcal{O}(y)$. Then, there are $t, s \in \mathbb{R}$, such that $z = \Phi(t,x) = \Phi(s,y)$. Thus, we have
\[
	x = \Phi\left(-t, \Phi(t,x)\right) = \Phi\left(-t,  \Phi(s,y)\right).
\]
Therefore, $x =  \Phi(s-t,y) \in \mathcal{O}(y)$, which implies $\mathcal{O}(x) \subset \mathcal{O}(y)$, Similarly, we have $y =  \Phi(t-s,x)$, and consequently $\mathcal{O}(y) \subset \mathcal{O}(x)$, Thus, we obtain $\mathcal{O}(x) = \mathcal{O}(y)$.
\end{proof}
\begin{remark}
By Lemma \ref{lem_disj}, the result of Theorem \ref{thm_mono} also holds for local flows defined in a neighborhood of the open set $A$ where the first-out map is considered. This possibility of local consideration is one of the advantages of a first-exit map over a first-in map.
\end{remark}
Here, we consider the next class of hybrid systems, which is an adaptation of the definition appearing in \cite{bernardo2008piecewise}.
\begin{definition}\label{def_imp}
An \emph{impacting system} is a system defined by a triple of a local flow $\Phi$ defined in the neighborhood of $H^+ :=\{(x,y) \mid y > 0\}\subset \mathbb{R}^2$, a partial map $P: \mathbb{R} \to \mathbb{R}$ such that
\[
	\mathrm{Im}\, E_{H^+} \subset \mathrm{dom}\, P \times \{0\},
\]
and a local flow $\Phi_s:\mathbb{R}\times R_s \to R_s$, where \[R_s := \{(x, 0) \mid P(x) = x \text{ and } E_{H^{+}} (x,0) = (x, 0)\}.\]

We denote an impacting system as $(P, \Phi, \Phi_s)$.

Two impacting systems $(P, \Phi, \Phi_s)$ and $(Q, \Psi, \Psi_s)$ are \emph{topologically conjugate} if there exists a homeomorphism $H: \bar H^{+} \to \bar H^{+} $ such that
\[
	\begin{aligned}
		&H( P(x), 0) = (Q, 0)\left(H(x, 0)\right) \\
		&\Psi\left(t, H(x,y)\right) = H\left(\Phi(t,x,y)\right)\\
		&\Psi_s\left(t, H(x,0)\right) = H\left(\Phi_s(t,x,0)\right)
	\end{aligned}
\]
whenever these expressions are defined for $t$, $x$, and $y$.  Here we define $(Q, 0)(x,y):= \qty(Q(x), 0)$.
\end{definition}
 \begin{remark}
 The local flow $\Phi_s$ describes the sliding mode of the system.
 \end{remark}
 
The definition of orbits of an impacting system requires an additional notion. Let $S$ and $S' \subset S$ be well-ordered sets such that $\mathrm{succ}\,(S') \subset S$, where $\mathrm{succ}\,$ is the successor function. Then, a \emph{hybrid time domain} is a set $\mathcal{T}$ of the form
\[
	\mathcal{T} := \bigcup_{n \in S'} \{n\}\times[t(n), t\left(\mathrm{succ}\,(n)\right)) \subset  S \times \mathbb{R}_{\geq 0},
\] 
where $t : S' \to  \mathbb{R}_{\geq 0}\cup \{\infty\}$ is an \emph{event time sequence} defined to be an order-preserving function, such that $t(\min S') =0$ and $t\qty(\mathrm{succ}\,(n)) = \infty$ implies $n$ is the maximal element of $S'$. In addition, we assume $[t(n), t\qty(\mathrm{succ}\,(n)) = \{t(n)\}$ if $t(n) =  t\qty(\mathrm{succ}\,(n))$. This definition of the hybrid time domain is a modification of the formulation in \cite{collins2006generalised}.

A \emph{forward trajectory} of an impacting system $(P, \Phi, \Phi_s)$ from $(x_0, y_0) \in \mathbb{R}^2$ is a map $\gamma$ from a hybrid time domain $\mathcal{T}$ to $\mathbb{R}^2$ such that
\begin{enumerate}
	\item $\gamma(\min S' , 0 ) = (x_0, y_0).$
	\item For each $s \in S',$ the dynamics on $[t(s), t\qty(\mathrm{succ}\, s))$ is described either by $\Phi$ or $\Phi_s$, i.e., if $\gamma\qty(s, t(s))\in R_s$, 
	\[
		\gamma(s, \tau) = \Phi_s\qty(\tau - t(s), \gamma\qty(s, t(s)))
	\]
	for all $\tau \in [t(s), t\qty(\mathrm{succ}\, s))$, and if $\gamma\qty(s, t(s))\not\in R_s,$ 
	\[
		\gamma(s, \tau) = \Phi\qty(\tau - t(s), \gamma\qty(s, t(s)))
	\]
	for all $\tau \in [t(s), t\qty(\mathrm{succ}\, s))$.
	\item Event times of $\gamma$ are jumping times, i.e.,
		\[
			\gamma\qty(\mathrm{succ}\,s, t\qty(\mathrm{succ}\, s)) = (P, 0) \left(\lim_{\tau \to t(\mathrm{succ}\,s)}\gamma(s, \tau)\right) 
		\] 
\end{enumerate}
This definition of forward trajectory is an adaptation of the concept of the \emph{solution} in \cite{goebel2012hybrid}.
The assumption that the index set $S$ is well-ordered is essential for a rule-based description of a model, as it enables specifying the dynamics in terms of updating rules. 
\begin{remark}
In our definition of forward trajectories, uniqueness is not guaranteed. For example, it is possible for trajectories starting from a point on $\mathbb{R}\times\{0\}$ to move either by a flow or a map.
\end{remark}
\begin{remark}
There is no simultaneous multiple jumping in an impact system. If such an event occurs at $(x,0)$, it follows that $P(x) = x$ and $E_{H^{+}} (x,0) = (x, 0)$. Therefore $(x, 0)\in R_s$ and its orbit remains in $R_s$ for some time because the domain of $\Phi_s$ is required to be open.
\end{remark}
We introduce the following notion to consider the connection between hybrid systems and local flows.
\begin{definition}
A pair of local flows $\Phi_1$ and $\Phi_2$ \emph{induces an impacting system} if $\Phi_1$ and $\Phi_2$ are defined in the neighborhoods of $H^-$ and $H^+$ respectively, and $(P, \Phi_2, \Phi_s)$ is an impacting system. Here, $E_{H^{+}} = \qty(P(x), 0)$ is the first-out map of $\Phi_1$ and $\Phi_s$ is the restriction of $\Phi_2$ to $R_s$.
\end{definition}
Impacting systems of physical origin are often induced by local flows. By describing the transient dynamics of the switching using a flow, we may obtain a representation of the partial map of reset.

\begin{example}
 A typical example of an impacting system is the \emph{impact oscillator}, which is a harmonic oscillator with the influence of floor considered. In our definition, it can be formulated as an impacting system $\qty(P, \Phi, \Phi_s)$ defined by $P(x) = \max\qty(-\mu x,x)$ for some $\mu>0$, 
\[
	\Phi(t, x,y) = \qty(\cos(\pi t) x -\sin(\pi t)y, \sin(\pi t) x + \cos(\pi t) y) 
\]
and $\Phi_s(t, x, y) = (x,y)$. By the construction in Theorem \ref{thm_conti}, we can show that it is induced by a combination of local flows.
\end{example}
However, it is not necessarily true that every flow induces an impact system.
\begin{example}\label{ex_comsin}
Let $h:\mathbb{R} \to \mathbb{R}$ be defined by
\[
	h(x) := \begin{cases}
				x\sin(1/x) & x \neq 0\\
				0 & x= 0.
			\end{cases}
\]

Since $h$ is continuous, $H:\mathbb{R}^2 \to \mathbb{R}^2$ defined by 
\[
	H(x,y) := (x, y+h(x))
\]
is a homeomorphism. Let $\Phi$ be the flow induced from $H$ and $\Psi(t, x, y) := (x +t, y)$ by Lemma \ref{lem_mkflow}. Then, it does not induce an impact system. This is because the timing of the switching is not well-ordered for the Zeno trajectory of the origin, and consequently, a map cannot specify the ``next" point of the origin. 
\end{example}

\begin{remark}
The composition of the first-out map and the reset map can be regarded as a Poincar\'e map. In particular, it is easily observed that the fixed points of the composite map correspond to the equilibrium or periodic points of the impacting system.
\end{remark}

\begin{theorem}
If flows $\Phi$ and $\Psi$ are topologically conjugate and induce impacting systems, induced systems are topologically conjugate.
\end{theorem}
\begin{proof}
This follows immediately from Lemma \ref{lem_conj}.
\end{proof}

The importance of impact oscillators can be understood by the following theorem. According to this result, the resetting map can be taken to be that of the impact oscillator for a rather broad class of impacting systems. This result also implies that the map part of an impact system takes a rather limited form if they are derived from flows and sufficiently well-behaved.

\begin{theorem}[Main Theorem D]
Let $(P, \Phi, \Phi_s)$ be an impacting system induced by local flows. If $P$ is total, continuous, and not identity, then $(P, \Phi, \Phi_s)$ is topologically conjugate with another impacting system $(Q, \Psi, \Psi_s)$, where $Q(x) = -x$ if $x \leq 0$.
\end{theorem}
\begin{proof}
Let the map $(P, 0) $ be the first-out map of a local flow $\tilde\Phi$. By an argument similar to that in Corollary \ref{cor_circ}, we can ascribe $\tilde\Phi$ to another local flow defined in a neighborhood of the unit disc so that infinity is mapped to 1, which is an equilibrium. Then, we apply Corollary \ref{cor_con} and Remark \ref{rem_uni} to conclude that $P$ is a unimodal map. By flipping the x-axis, we may assume that $P$ decreases monotonically on $(-\infty, \alpha]$ for some $\alpha \in \mathbb{R}$ and identity on $[\alpha, \infty).$ By shifting the plane by $\alpha$, we may assume that $\alpha =0$ without loss of generality.

Let us define a homeomorphism $h: \mathbb{R} \to \mathbb{R}$ by 
\[
	h(x) = \begin{cases}
				-P^{-1}(x) & (x \geq 0)\\
				x & (x < 0),
			\end{cases}
\]
where $P^{-1}$ denotes the negative branch. Then, $H(x,y) := (h(x), y)$ gives the desired conjugation.
\end{proof}
When an impacting system is induced by flows, Lemma \ref{lem_mono_F} restricts the possible behavior. 
\begin{theorem}
Let $(P, \Phi, \Phi_s)$ be an impacting system induced by local flows, $P(x) = y$ and $E_{H^+}(y,0) = (z, 0)$. If $z$ lies between $x$ and $y$, then all intersections of the forward orbit of $(z,0)$ with $\mathbb{R}\times \{0\}$ lie between $x$ and $y$.
\end{theorem}
\begin{proof}
This is an immediate consequence of Lemma \ref{lem_mono_F}.
\end{proof}

\section{Concluding Remarks}
We have defined the concept of the first-out and first-in maps and considered their basic properties. Although the discussion here will serve as a proof of concept, it is clear that there are issues to be considered. 

In particular, the global restriction of type sequences in two-dimensional flows poses interesting questions. It is important to identify whether there is any prohibited combination of types other than BC because it gives us information on the dynamics of planar flow in general. If this exists, such restriction will be global, and a detailed analysis will be required.

\section*{Acknowledgements}
This study was supported by a Grant-in-Aid for JSPS Fellows (20J01101). 
\bibliographystyle{plain}
\bibliography{reduction_flow}
\end{document}